\documentclass[11pt,a4paper]{amsart}

\usepackage{amsmath, amsthm, amssymb}

\newcommand{\bP}{\mathbb{P}}
\newcommand{\sP}{\mathsf{P}}
\newcommand{\sA}{\mathsf{A}}
\newcommand{\cS}{\mathcal{S}}
\newcommand{\bN}{\mathbb{N}}
\newcommand{\PGL}{\mathrm{PGL}}
\newcommand{\cO}{\mathcal{O}}
\newcommand{\supp}{\operatorname{supp}}
\newcommand{\Id}{\mathrm{Id}}
\newcommand{\bR}{\mathbb{R}}
\newcommand{\can}{\operatorname{can}}

\newcommand{\ord}{\operatorname{ord}}
\newcommand{\Vu}{\boldsymbol{u}}
\newcommand{\Vv}{\boldsymbol{v}}
\newcommand{\sF}{\mathsf{F}}
\newcommand{\sJ}{\mathsf{J}}
  
\numberwithin{equation}{section}
\theoremstyle{plain}
\newtheorem{theorem}{Theorem}[section]
\newtheorem{lemma}[theorem]{Lemma}

\newtheorem{mainth}{Theorem}

\theoremstyle{definition}

\newtheorem*{acknowledgement}{Acknowledgement}

\theoremstyle{remark}
\newtheorem{remark}[theorem]{Remark}

\begin{document}  
  
\title[An a priori bound of rational functions]{An a priori bound of 
rational functions on the Berkovich projective line}

\author[Y\^usuke Okuyama]{Y\^usuke Okuyama}
\address{
Division of Mathematics,
Kyoto Institute of Technology,
Sakyo-ku, Kyoto 606-8585 Japan}
\email{okuyama@kit.ac.jp}

\date{\today}

\begin{abstract}
We establish a locally uniform a priori bound on the dynamics of 
a rational function $f$ of degree $>1$ on the Berkovich projective
line over an algebraically closed field of any characteristic
that is complete with respect to a non-trivial and non-archimedean
absolute value, and deduce an equidistribution result for moving targets
towards the equilibrium (or canonical) measure $\mu_f$, 
under the no potentially good reductions condition. This partly answers
a question posed by Favre and Rivera-Letelier. 
\end{abstract}

\subjclass[2010]{Primary 37P50; Secondary 11S82}
\keywords{a priori bound, domaine singulier, equidistribution, moving targets,
no potentially good reductions, 
Berkovich projective line} 

\maketitle

\section{Introduction}\label{sec:intro}

Let $K$ be an algebraically closed field of {\itshape any} characteristic
that is complete with respect to a non-trivial and non-archimedean
absolute value $|\cdot|$. 
The {\itshape Berkovich} projective line $\sP^1=\sP^1(K)$   
compactly augments the {\itshape classical} projective line
$\bP^1=\bP^1(K)$ (see \cite{Berkovichbook})
and is canonically regarded as a tree in the sense
of Jonsson \cite[Definition 2.2]{Jonsson15}, the topology of which
coincides with the (Gelfand) topology of $\sP^1$. 
The continuous action on $\bP^1$ of 
a rational function $h\in K(z)$ canonically extends to
that on $\sP^1$. If in addition $\deg h>0$, 
then this extended continuous action is also open and surjective, 
preserves both $\bP^1$ and $\sP^1\setminus\bP^1$, and 
satisfies $\# h^{-1}(\cS)\le\deg h$ for every $\cS\in\sP^1$.
More precisely, 
the local degree function $\deg_{\cdot}h:\bP^1\to\{1,\ldots,\deg h\}$
also canonically extends to an upper semicontinuous function 
$\sP^1\to\{1,\ldots,\deg h\}$ 
so that $\sum_{\cS'\in h^{-1}(\cS)}\deg_{\cS'}h=\deg h$ 
for every $\cS\in\sP^1$, 
and induces the pullback action $h^*$ of $h$
on the space of all Radon measures on $\sP^1$; 
letting $\delta_{\cS}$ be the Dirac measure on $\sP^1$
at each $\cS\in\sP^1$, 
$h^*\delta_{\cS}=\sum_{\cS'\in h^{-1}(\cS)}(\deg_{\cS'}h)\delta_{\cS'}$
on $\sP^1$.

By the seminal Baker--Rumely \cite{BR10}, 
Chambert-Loir \cite{ChambertLoir06}, and Favre--Rivera-Letelier \cite{FR09},
for every $f\in K(z)$ of degree $d>1$, there is the unique
{\itshape equilibrium $($or canonical$)$ measure} 
$\mu_f$ of $f$ on $\sP^1$, 
which has no masses on polar subsets in $\sP^1$ and  
satisfies the $f$-balanced property 
\begin{gather*}
 f^*\mu_f=d\cdot\mu_f\quad\text{on }\sP^1
\end{gather*}
(so in particular the $f$-invariance $f_*\mu_f=\mu_f$ on $\sP^1$).
Moreover, the equidistribution for {\itshape iterated pullbacks of points}
\begin{gather}
\lim_{n\to\infty}\frac{(f^n)^*\delta_{\cS}}{d^n}=\mu_f\quad\text{weakly on }\sP^1\label{eq:pullbacknonarch}
\end{gather}
holds for every $\cS\in\sP^1$ but the {\itshape at most countable} exceptional set 
$E(f):=\{a\in\bP^1:\#\bigcup_{n\in\bN\cup\{0\}}f^{-n}(a)<\infty\}$ 
of $f$; in particular, $\mu_f$ is mixing (so ergodic) under $f$.
If $\operatorname{char}K=0$, then $\#E(f)\le 2$.
In general, $E(f)$ is at most countable and $\bP^1\setminus E(f)$
is dense in $\bP^1$.

Our aim is to contribute to a locally uniform quantitative study 
and an equidistribution problem on the dynamics of $f$ on $\sP^1$.

\subsection{An {\itshape a priori} bound of the dynamics of $f$}\label{sec:boundnon-arch}\label{sec:bound}

Recall 
(that the absolute value $|\cdot|$ is said to be non-trivial 
if $|K|\not\subset\{0,1\}$ and)
that the absolute value $|\cdot|$ is said to be non-archimedean if
the {\itshape strong triangle inequality}
\begin{gather*}
 |z+w|\le\max\{|z|,|w|\}\quad\text{for any }z,w\in K 
\end{gather*}
holds. The (normalized) chordal metric $[z,w]_{\bP^1}$ on $\bP^1=\bP^1(K)$ 
(see \eqref{eq:chordalnonarch} for the definition. 
The notation is adopted from
Nevanlinna's and Tsuji's books \cite{Nevan70,Tsuji59})
is in particular written as
\begin{gather*}
 [z,w]_{\bP^1}=\frac{|z-w|}{\max\{1,|z|\}\max\{1,|w|\}}(\le 1)
\end{gather*}
for any $z,w\in K=\bP^1\setminus\{\infty\}$.
The projective transformations group on $\bP^1$ is identified
with $\PGL(2,K)$.

Let $f\in K(z)$ be a rational function on $\bP^1$ of degree $d>1$.
We say that $f$ {\em has no potentially good reductions}\footnote{We 
avoided the terminology ``{\em potential} good reduction'' since
it does not regard {\em potential} theory, which is one of main tools
in this article.}
if for every $h\in\PGL(2,K)$, $\deg(\widetilde{h\circ f\circ h^{-1}})<\deg f$,
where $\widetilde{h\circ f\circ h^{-1}}\in k(z)$ 
(of degree $\le\deg f$) denotes the {\itshape reduction} of 
$h\circ f\circ h^{-1}$ 
modulo the unique maximal ideal $\mathfrak{m}_K:=\{z\in K:|z|<1\}$ 
of the ring $\cO_K:=\{z\in K:|z|\le 1\}$ of $K$-integers, and
$k:=\cO_K/\mathfrak{m}_K$ is the residue field of $K$.
It is known that {\em $f$ has no potentially good reductions if and only if
$\mu_f(\{\cS\})=0$ for any $\cS\in\sP^1\setminus\bP^1$}
(cf.\ \cite[Corollary 10.33]{BR10}).

Our principal result is 
the following locally uniform {\itshape a priori} bound of the dynamics of $f$
{\em for moving targets}, under the no potentially good reductions condition.

\begin{mainth}\label{th:apriorinonarch}
Let $K$ be an algebraically closed field of any characteristic
that is complete with respect to a non-trivial and non-archimedean
absolute value. Then for every rational function $f\in K(z)$ on $\bP^1$
of degree $d>1$ having no potentially good reductions, 
every rational function $g\in K(z)$ on $\bP^1$ 
of degree $>0$, and every non-empty open subset $D$ in $\bP^1$, we have
\begin{gather}
\lim_{n\to\infty}\frac{\sup_{w\in D}\log[f^n(w),g(w)]_{\bP^1}}{d^n+\deg g}=0.\label{eq:apriorinonarchi}
\end{gather}
\end{mainth}
The argument in the proof is similar to
those in Buff--Gauthier \cite{BG14} and Gauthier \cite{Gauthier13},
using a comparison principle
(Bedford--Taylor \cite{BT87}; see also Bedford--Smillie \cite[Page 77]{BS91})
from pluripotential theory.
We note that if $g\equiv a$ on $\bP^1$ for some $a\in\bP^1$,
then \eqref{eq:apriorinonarchi} still holds unless $a\in E(f)$. 
The choice of $d^n+\deg g=d^n+o(d^n)$
as $n\to\infty$ rather than the simpler $d^n$ 
as the denominator in \eqref{eq:apriorinonarchi}
reflects the fact that whenever $f^n\not\equiv g$, 
the degree of the effective divisor $[f^n=g]$
on $\bP^1$ defined by all roots in $\bP^1$ of the algebraic equation $f^n=g$
taking into account their multiplicities equals $d^n+\deg g$.

\subsection{Equidistribution towards $\mu_f$ for moving targets}

For every $g\in K(z)$ of degree $>0$ and every $n\gg 1$,
we regard the effective divisor $[f^n=g]$ on $\bP^1$ mentioned at the end of
Subsection \ref{sec:bound} as the positive Radon measure 
\begin{gather*}
 \sum_{a\in\bP^1}(\ord_a[f^n=g])\delta_a\quad\text{on }\sP^1, 
\end{gather*}
which we also denote by the same notation $[f^n=g]$ for simplicity.
Then $[f^n=g]/(d^n+\deg g)$ is a {\itshape probability} Radon measure on $\sP^1$.

The following equidistribution {\itshape for moving targets} 
under the no potentially good reductions condition
is an application of Theorem \ref{th:apriorinonarch}, and
partly answers the question posed by Favre--Rivera-Letelier
\cite[apr\`es Th\'eor\`eme B]{FR09}.

\begin{mainth}\label{th:equidistpositive}
Let $K$ be an algebraically closed field of any characteristic
that is complete with respect to a non-trivial and non-archimedean
absolute value. Then for every $f\in K(z)$
of degree $d>1$ having no potentially good reductions and 
every $g\in K(z)$ of degree $>0$, we have
\begin{gather}
\lim_{n\to\infty}\frac{[f^n=g]}{d^n+\deg g}=\mu_f\quad\text{weakly on }\sP^1.\label{eq:equidistmoving}
\end{gather}
\end{mainth}

In \cite[Th\'eor\`eme B]{FR09},
they established \eqref{eq:equidistmoving} 
in the $\operatorname{char}K=0$ case (even without the no potentially good
reductions assumption), and asked about the situation
in the $\operatorname{char}K>0$ case.
In Theorem \ref{th:equidistpositive}, in the $\operatorname{char}K>0$
case, the no potentially good reductions assumption can be relaxed
but cannot be omitted (e.g., $f(z)=z+z^p$ and $g(z)=z$
where $p=\operatorname{char}K>0$, as pointed out in \cite[apr\`es Th\'eor\`eme B]{FR09}).

\subsection{Organization of the article}
In Section \ref{sec:background}, we recall background on the
topology, potential theory, and dynamics on the Berkovich projective line.
In Section \ref{sec:key}, we show a lemma, which plays a key role
in the proof of Theorem \ref{th:apriorinonarch}.
In Sections \ref{sec:apriorinonarch} and \ref{sec:equidistnonarch}, 
we show Theorems \ref{th:apriorinonarch} and \ref{th:equidistpositive},
respectively.

\section{Background}
\label{sec:background}

Let $K$ be an algebraically closed field of any characteristic
that is complete with respect to a non-trivial and non-archimedean
absolute value $|\cdot|$. 

\subsection{The Berkovich projective line}

The Berkovich affine line $\sA^1=\sA^1(K)$ is the set of all
multiplicative seminorms on $K[z]$ which restricts to $|\cdot|$
on $K(\subset K[z]$ naturally). 
We write an element of $\sA^1$
like $\cS$, and denote it by $[\cdot]_{\cS}$ as a multiplicative
seminorm on $K[z]$. Under this convention, $\sA^1$ is
equipped with the weakest topology (the Gelfand topology)
such that for any $\phi\in K[z]$, 
the function $\sA^1\ni\cS\mapsto[\phi]_{\cS}\in\bR_{\ge 0}$ is continuous,
and then $\sA^1$ is a locally compact and Hausdorff topological
space. For more details, see \cite{Berkovichbook} (for the full generality
of Berkovich analytic spaces),
\cite{BR10}, \cite{FR09} (for the details on $\sP^1$).

Let us call a subset $B$ in $K$ a {\em $K$-closed disk} if
for some $a\in K$ and some $r\ge 0$, $B=\{z\in K:|z-a|\le r\}$. 
For any $K$-closed disks $B$ and $B'$, {\em if $B\cap B'\neq\emptyset$, then 
either $B\subset B'$ or $B'\subset B$}
(by the strong triangle inequality). 
The Berkovich representation \cite{Berkovichbook} asserts that
{\em any element $\cS\in\sA^1$ is 
induced by a non-increasing and nesting sequence $(B_n)$ of
$K$-closed disks $B_n$ in that}
\begin{gather*}
 [\phi]_{\cS}=\inf_{n\in\bN}\sup_{z\in B_n}|\phi(z)|
\quad\text{for any }\phi\in K[z]. 
\end{gather*}
In particular,
each point in $K=\bP^1\setminus\{\infty\}$ and, 
more generally, each $K$-closed disk are regarded as elements of $\sA^1$.

We would need some details on the topology of $\sP^1$ later, so let us introduce
$\sP^1$ as an ``$\bR$-''tree in the sense of Jonsson \cite[Definition 2.2]{Jonsson15}
as follows.
Any $[\cdot]_{\cS}\in\sA^1$ extends to the function
$K(z)\to\bR_{\ge 0}\cup\{+\infty\}$ 
such that for any $\phi=\phi_1/\phi_2\in K(z)$ where
$\phi_1,\phi_2\in K[z]$ are coprime,
$[\phi]_{\cS}=[\phi_1]_{\cS}/[\phi_2]_{\cS}\in\bR_{\ge 0}\cup\{+\infty\}$, 
and we also regard $\infty\in\bP^1$ as the function 
$[\cdot]_{\infty}:K(z)\to\bR_{\ge 0}\cup\{+\infty\}$ 
such that for every $\phi\in K(z)$,
$[\phi]_{\infty}=|\phi(\infty)|\in\bR_{\ge 0}\cup\{+\infty\}$.

As a set, let us define $\sP^1:=\sA^1\cup\{\infty\}$,
which is also equipped with a (partial) order 
$\le_\infty$ so that for any $\cS,\cS'\in\sP^1$, $\cS\le_\infty\cS'$
if and only if $[\cdot]_{\cS}\le[\cdot]_{\cS'}$ on $K[z](\subset K(z))$.
For any $\cS,\cS'\in\sP^1$, if $\cS\le_{\infty}\cS'$,
then set $[\cS,\cS']=[\cS',\cS]:=\{\cS''\in\sP^1:\cS\le\cS''\le\cS'\}$, and
in general, 
there is the unique, say, $\cS\wedge_\infty\cS'\in\sP^1$ such that
$[\cS,\infty]\cap[\cS',\infty]=[\cS\wedge_\infty\cS',\infty]$, and
set $[\cS,\cS']:=[\cS,\cS\wedge_\infty\cS']\cup[\cS\wedge_\infty\cS',\cS']$.

For any $\cS\in\sP^1$, let us introduce the equivalence class 
\begin{gather*}
 T_{\cS}\sP^1:=(\sP^1\setminus\{\cS\})/\sim,
\end{gather*}
where for every $\cS',\cS''\in\sP^1\setminus\{\cS\}$, we say
$\cS'\sim\cS''$ if $[\cS,\cS']\cap[\cS,\cS'']\neq\{\cS\}$, or equivalently,
$[\cS,\cS']\cap[\cS,\cS'']=[\cS,\cS''']$ for some $\cS'''\in\sP^1\setminus\{\cS\}$. An element of $T_{\cS}\sP^1$ represented by an element 
$\cS'\in\sP^1\setminus\{\cS\}$ is denoted by $\overrightarrow{\cS\cS'}$. 
We call an element of $T_{\cS}\sP^1$ a {\itshape direction} of $\sP^1$
at $\cS$ and write it like $\Vv$, and also denote it by $U(\Vv)$ as a
subset in $\sP^1\setminus\{\cS\}$; 
we note that for every $a\in\bP^1$, $\#T_a\sP^1=1$. 

Let us equip $\sP^1$ with the {\em observer}
topology having the quasi-open basis
$\bigl\{U(\Vv):\cS\in\sP^1,\Vv\in T_{\cS}\sP^1\bigr\}$. 
This topological space $\sP^1$ coincides with the one-point compactification
of $\sA^1$, 
both $\bP^1$ and $\sP^1\setminus\bP^1$ are dense in $\sP^1$,
the set $U(\Vv)$ is a component of $\sP^1\setminus\{\cS\}$
for each $\cS\in\sP^1$ and each $\Vv\in T_{\cS}\sP^1$,
and for any $\cS,\cS'\in\sP^1$,
the set $[\cS,\cS']$ is an arc in $\sP^1$.
 
We prepare the following.

\begin{lemma}\label{th:boundary}
 For any domains $U,V$ in $\sP^1$, if $U\cap V=\emptyset$ and
 $\partial U\cap\partial V\neq\emptyset$, then 
 $\partial U\cap\partial V$ is a singleton, say, $\{\cS_0\}$ in
 $\sP^1\setminus\bP^1$, and moreover, there are
 distinct $\Vu,\Vv\in T_{\cS_0}\sP^1$ such that
 $U\subset U(\Vu)$ and $V\subset U(\Vv)$.
\end{lemma}

\begin{proof}
 Let $U,V$ be domains in $\sP^1$ satisfying $U\cap V=\emptyset$ and
 $\partial U\cap\partial V\neq\emptyset$, and
 fix $\cS_0\in\partial U\cap\partial V$. Then
 $U\subset U(\Vu)$ and $V\subset U(\Vv)$
 for some $\Vu,\Vv\in T_{\cS_0}\sP^1$ since
 both $U$ and $V$ are domains in $\sP^1\setminus\{\cS_0\}$.
 Then we must have $\Vu\neq\Vv$, or equivalently $U(\Vu)\cap U(\Vv)=\emptyset$ 
 since for some (indeed any) $\cS\in U$ and some (indeed any) $\cS'\in V$,
 we have $\Vu=\overrightarrow{\cS_0\cS}, 
 \Vv=\overrightarrow{\cS_0\cS'},[\cS_0,\cS]\setminus\{\cS_0\}\subset U$,
 $[\cS_0,\cS']\setminus\{\cS_0\}\subset V$ (and the assumption
 $U\cap V=\emptyset$). Now the proof is complete since
 $(\partial U)\setminus\{\cS_0\}\subset U(\Vu)$ and
 $(\partial V)\setminus\{\cS_0\}\subset U(\Vv)$.
\end{proof}

\subsection{Potential theory on $\sP^1$}
For potential theory on $\sP^1=\sP^1(K)$ including 
the fully general study of harmonic analysis on $\sP^1$, i.e., subharmonic functions on open subsets in $\sP^1$, see 
Baker--Rumely \cite{BR10} and Thuillier \cite{ThuillierThesis},
and for the study of the class of 
``$\delta_{\cS_{\can}}$-subharmonic functions'' 
$g:\sP^1\to\bR_{\le 0}\cup\{-\infty\}$ 
such that
\begin{gather*}
 \Delta g+\delta_{\cS_{\can}} 
\end{gather*}
are probability Radon measures on $\sP^1$,
see Favre--Rivera-Letelier \cite{FR09}.
Here 
$\Delta=\Delta_{\sP^1}$ is the Laplacian on $\sP^1$
(in \cite{BR10} the opposite sign convention on $\Delta$ is adopted), and 
the {\itshape Gauss $($or canonical$)$ point} 
$\cS_{\can}\in\sP^1$ is {\em represented} 
by the (constant sequence of the) $K$-closed disk
$\cO_K=\{z\in K:|z|\le 1\}$, i.e., for every $\phi\in K[z]$,
$[\phi]_{\cS_{\can}}=\sup_{z\in\cO_K}|\phi(z)|$.

Let $\|(z_0,z_1)\|=\max\{|z_0|,|z_1|\}$ be the {\itshape maximal} norm 
on $K^2$, and $\pi:K^2\setminus\{(0,0)\}\to\bP^1$ be the canonical projection. 
Noting that $K^2\wedge K^2\cong K$, the (normalized) chordal metric 
$[z,w]$ on $\bP^1$ is defined as
\begin{gather}
 [z,w]_{\bP^1}:=\frac{|Z\wedge W|}{\|Z\|\cdot\|W\|}\le 1,\quad z,w\in\bP^1,\label{eq:chordalnonarch}
\end{gather}
where $Z\in\pi^{-1}(z)$ and $W\in\pi^{-1}(w)$;
the topology
on $(\bP^1,[z,w])$ coincides with the relative topology of $\bP^1$ as a subset
of $\sP^1$. 
The function 
\begin{gather*}
 \log\max\{1,|\cdot|\}(=-\log[\cdot,\infty]_{\bP^1})\quad\text{on }K
\end{gather*}
extends {\em continuously} to $\sP^1\setminus\{\infty\}$
and in turn to a function $\sP^1\to\bR_{\ge 0}\cup\{+\infty\}$ 
such that, writing also this extension
as $\log\max\{1,|\cdot|\}$ for simplicity, we have
\begin{gather}
 \Delta(\log\max\{1,|\cdot|\})=\delta_{\cS_{\can}}-\delta_\infty
\quad\text{on }\sP^1.\label{eq:Laplacianaffine} 
\end{gather}

\subsection{Dynamics of rational functions on $\sP^1$}
Let $h\in K(z)$ be of degree $>0$.
A (non-degenerate homogeneous) {\itshape lift} of $h$ 
is an ordered pair $H=(H_0,H_1)\in(K[z_0,z_1]_{\deg h})^2$, 
which is unique up to multiplication in $K^*$, such that
$\pi\circ H=h\circ\pi$ on $K^2\setminus\{(0,0)\}$
(and that $H^{-1}(0,0)=\{(0,0)\}$).
Then the function
\begin{gather*}
 T_H:=\log\|H\|-(\deg h)\cdot\log\|\cdot\| 
\end{gather*}
on $K^2\setminus\{0\}$ descends to
$\bP^1$ and in turn extends {\em continuously} to $\sP^1$, and in fact
\begin{gather*}
 \Delta T_H=h^*\delta_{\cS_{\can}}-(\deg h)\cdot\delta_{\cS_{\can}}\quad\text{on }\sP^1\label{eq:potential}
\end{gather*}
(see, e.g., \cite[Definition 2.8]{OkuCharacterization}).

Let $f\in K(z)$ of degree $d>1$, and fix a lift $F$ of $f$. 
Then for every $n\in\bN$, $F^n$ is a lift of $f^n$ and
$\deg(f^n)=d^n$, and there is the uniform limit 
\begin{gather}
 g_F:=\lim_{n\to\infty}\frac{T_{F^n}}{d^n}\quad\text{on }\sP^1,\label{eq:Green} 
\end{gather}
which is continuous on $\sP^1$ and in fact satisfies 
\begin{gather}
 \Delta g_F=\mu_f-\delta_{\cS_{\can}}\quad\text{on }\sP^1\label{eq:canonicaldefining}
\end{gather}
(see \cite[\S 10]{BR10}, \cite[\S 6.1]{FR09}).
We call $g_F$ the {\em dynamical Green function of $F$ on} $\sP^1$.

For every $g\in K(z)$ of degree $>0$ and every $n\in\bN$,
the function $z\mapsto[f^n(z),g(z)]_{\bP^1}$ on $\bP^1$ 
extends continuously to a function 
\begin{gather*}
 \cS\mapsto [f^n,g]_{\can}(\cS)
\end{gather*}
on $\sP^1$ so that $0\le [f^n,g]_{\can}(\cdot)\le 1$ on $\sP^1$ and that
\begin{gather*}
 \Delta\log[f^n,g]_{\can}(\cdot)=[f^n=g]-(f^n)^*\delta_{\cS_{\can}}-g^*\delta_{\cS_{\can}}\quad\text{on }\sP^1
\end{gather*} 
(\cite[Proposition 2.9 and Remark 2.10]{OkuCharacterization});
this extension does not necessarily coincide with the evaluation function
$\cS\mapsto[\cS',\cS'']_{\can}$ at $(\cS',\cS'')=(f^n(\cS),g(\cS))\in(\sP^1)^2$.
Fixing also a lift $G$ of $g$, for every $n\in\bN$, we have
\begin{gather}
 \Delta(\log[f^n,g]_{\can}(\cdot)+T_{F^n}+T_{G})
=[f^n=g]-(d^n+\deg g)\delta_{\cS_{\can}}
\label{eq:rootsnormalized}
\end{gather} 
on $\sP^1$.

\subsection{The Fatou-Julia strategy}
Let $f\in K(z)$ be of degree $d>1$.
The Berkovich Julia set $\sJ(f)$ of $f$ is {\itshape defined by}
$\supp\mu_f$, so is non-empty, and the Berkovich 
Fatou set $\sF(f)$ of $f$ is by $\sP^1\setminus \sJ(f)$.
We prepare the following.

\begin{lemma}\label{th:Fatou}
Let $D$ be an open subset in $\sP^1$ such that, for some sequence $(n_j)$ 
in $\bN$ tending to $\infty$ as $j\to\infty$ and some $g\in K(z)$, 
$\lim_{j\to\infty}f^{n_j}=g$ uniformly as such mappings $D\cap\bP^1\to(\bP^1,[z,w])$. 
Then $D\subset \sF(f)$. 
\end{lemma}

\begin{proof}
Suppose to the contrary that there is $\cS\in D\cap \sJ(f)$. Then 
for any open neighborhood $D'$ of $\cS$ in $D$,
by the equidistribution \eqref{eq:pullbacknonarch}, 
$\liminf_{j\to\infty}f^{n_j}(D')$ contains the dense subset
$\bP^1\setminus E(f)$ in $\bP^1$,
and then $g(D'\cap\bP^1)$ must be dense in $\bP^1$. 
This is impossible since $\bP^1\setminus g(D'\cap\bP^1)$ 
contains a non-empty open subset in $\bP^1$ if $D'$ is small enough. 
\end{proof}

A Berkovich Fatou component $W$ of $f$ is a component of $\sF(f)$. 
Both $\sJ(f)$ and $\sF(f)$ are totally invariant under $f$,
$f$ maps a Berkovich Fatou component of $f$
properly to a Berkovich Fatou component of $f$,
and the preimage of a Berkovich Fatou component of $f$ under $f$ is the union of
(at most $d$) Berkovich Fatou components of $f$. 
A Berkovich Fatou component $W$ of $f$ is said to be
{\em cyclic} under $f$ if $f^p(W)=W$ for some $p\in\bN$, 
and then the minimal such $p$ is called the {\em exact period} of $W$ (under $f$). 
A cyclic Berkovich 
Fatou component $W$ of $f$ having the exact period, say, $p\in\bN$
is called a Berkovich {\itshape domaine singulier}
of $f$ if $f^p:W\to W$ is injective; then in particular
$f^{-1}(W)\neq W$ since $d>1$.

\section{A key lemma}
\label{sec:key}

Let $K$ be an algebraically closed field of any characteristic
that is complete with respect to a non-trivial and non-archimedean
absolute value $|\cdot|$. 

\begin{lemma}\label{th:boundarynonarch}
Let $f\in K(z)$ be of degree $d>1$, and 
have no potentially good reductions. Then
{\em (i)} for any Berkovich Fatou component $U$ of $f$,
we have $\partial U\neq \sJ(f)$
if $f^{-1}(U)\neq U$, and moreover, {\em (ii)} 
for every cyclic Berkovich Fatou component $W$ of $f$
satisfying $f^{-1}(W)\neq W$, we have $\mu_f(\partial U)=0$
for every component $U$ of $\bigcup_{n\in\bN\cup\{0\}}f^{-n}(W)$.
\end{lemma}

\begin{proof}
Let us see (i).
Let $U$ be a Berkovich Fatou component of $f$, and suppose not only
$f^{-1}(U)\neq U$ but, to the contrary, also $\partial U=\sJ(f)$. 
Pick a component $V$ of $f^{-1}(U)$ other than $U$. 
Then $U\cap V=\emptyset$ and $\partial U\cap\partial V=\sJ(f)\cap\partial V\neq\emptyset$, so by Lemma \ref{th:boundary} and $\partial U=\sJ(f)$,
there is $\cS_0\in\sP^1\setminus\bP^1$ such that 
\begin{gather*}
 \partial V=\{\cS_0\}\subset\partial U=f(\partial{V})=\{f(\cS_0)\},
\end{gather*}
and in turn $\supp\mu_f:=\sJ(f)=\partial U=\{\cS_0\}(=\{f(\cS_0)\})$. 
In particular, $\mu_f(\{\cS_0\})=1>0$, which contradicts the assumption that
$f$ has no potentially good reductions.

Let us see (ii).
Pick a cyclic Berkovich Fatou component $W$ of $f$
having the exact period $p\in\bN$ and
satisfying $f^{-1}(W)\neq W$. 
Then 
for any $n\in\bN$ and any distinct components
$U,V$ of $f^{-pn}(W)$, by Lemma \ref{th:boundary},
$\partial U\cap\partial V$
is either $\emptyset$ or a singleton in $\sP^1\setminus\bP^1$, 
the latter of which is still a $\mu_f$-null set
under the assumption that $f$ has no potentially good reductions.
Hence for every $n\in\bN$, also by the $f$-invariance of $\mu_f$ and $f^p(W)=W$,
we have
\begin{multline*}
\mu_f(\partial W)=\mu_f(f^{-pn}(\partial W))
=\sum_{U:\text{ a component of }f^{-pn}(W)}\mu_f(\partial U)\\
=\mu_f(\partial W)+\sum_{U:\text{ a component of }f^{-pn}(W)\text{ other than }W}\mu_f(\partial U),
\end{multline*}
which first concludes $\mu_f(\partial U)=0$ for every component $U$
of $f^{-pn}(W)$ other than $W$. In particular,
$\mu_f(\bigcup_{n\in\bN\cup\{0\}}f^{-pn}(\partial W))=\mu_f(\partial W)$,
which with the $f$-ergodicity of $\mu_f$ and $f^p(W)=W$ 
yields $\mu_f(\partial W)\in\{0,1\}$. 
Under the assumption that $f^{-1}(W)\neq W$,
this with (i) also concludes $\mu_f(\partial W)=0$.
\end{proof}

\section{Proof of Theorem \ref{th:apriorinonarch}}
\label{sec:apriorinonarch}

Let $K$ be an algebraically closed field of any characteristic
that is complete with respect to a non-trivial and non-archimedean
absolute value $|\cdot|$. Let $f\in K(z)$ be of degree $d>1$.

Suppose now that there are $g\in K(z)$ of degree $>0$ and
a non-empty open subset $D$ in $\bP^1$ 
such that \eqref{eq:apriorinonarchi} does not hold, or equivalently,
replacing $D$ with some component of the minimal open subset in $\sP^1$
containing the original $D$ as the dense subset,
there is a sequence $(n_j)$ in $\bN$ tending to $\infty$ as $j\to\infty$
such that
\begin{gather}
\lim_{j\to\infty}\frac{\sup_{\cS\in D}\log[f^{n_j},g]_{\can}(\cS)}{d^{n_j}+\deg g}<0.\label{eq:quicknonarch}
\end{gather}
Then $D\subset \sF(f)$ by Lemma \ref{th:Fatou}.
Let $U$ be the Berkovich Fatou component of $f$
containing $D$. Since $\deg g>0$, 
$\lim_{j\to\infty}f^{n_{j+1}-n_j}=\Id_{g(D)\cap\bP^1}$ 
uniformly as such mappings $g(D)\cap\bP^1\to(\bP^1,[z,w])$, 
and then there exists $N\in\bN\cup\{0\}$ such that 
$V:=f^{n_N}(U)(\supset g(D))$ is a {\em cyclic} Berkovich
Fatou component of $f$ (having the exact period, say, $p\in\bN$). 
Also by Rivera-Letelier's counterpart of
Fatou's classification of cyclic (Berkovich) Fatou components
(\cite[Proposition 2.16 and its {\itshape esquisse de d\'emonstration}]{FR09},
see also \cite[Remark 7.10]{Benedetto10}), in fact $V$ is a 
Berkovich {\itshape domaine singular} of $f$ (so $f^{-1}(V)\neq V$).

We have the uniform upper bound 
\begin{gather*}
 \sup_{j\in\bN}\sup_{\cS\in\sP^1}\log[f^{n_j},g]_{\can}(\cS)\le 0. 
\end{gather*}
Moreover, we also have the lower bound 
\begin{gather*}
 \limsup_{j\to\infty}
\frac{\sup_{\cS\in\sP^1}\log[f^{n_j},g]_{\can}(\cS)}{d^{n_j}+\deg g}\ge 0>-\infty; 
\end{gather*}
for, otherwise, by Lemma \ref{th:Fatou}, we must have $\sP^1=\sF(f)$. This contradicts
$\sJ(f):=\supp\mu_f\neq\emptyset$.

Hence, recalling also the uniform convergence \eqref{eq:Green} and 
the equalities \eqref{eq:canonicaldefining} and \eqref{eq:rootsnormalized},
by a version of Hartogs's lemma (cf.\ \cite[Proposition 2.18]{FR09}, \cite[Proposition 8.57]{BR10}), we can assume that there is 
a function $\phi:\sP^1\to\bR_{\le 0}\cup\{-\infty\}$ such that
\begin{gather*}
 \phi=\lim_{j\to\infty}\frac{\log[f^{n_j},g]_{\can}(\cdot)}{d^{n_j}+\deg g}
\quad\text{on }\sP^1\setminus\bP^1
\end{gather*}
and that $\Delta\phi+\mu_f=\Delta(\phi+g_F)+\delta_{\cS_{\can}}$
is a {\em probability} Radon measure on $\sP^1$,
with no loss of generality
(see \cite[\S3.4]{FR09}
for a similar computation). Since, also by \eqref{eq:Laplacianaffine}, we have 
\begin{gather*}
 \Delta(\phi+g_F+\log\max\{1,|\cdot|\})+\delta_\infty
=\Delta(\phi+g_F)+\delta_{\cS_{\can}}\ge 0\quad\text{on }\sP^1,
\end{gather*}
the function 
\begin{gather*}
 \phi(=(\text{a subharmonic function})-(\text{a continuous function})
\text{ on }\sP^1\setminus\{\infty\})
\end{gather*}
is {\em upper semicontinuous} on $\sP^1\setminus\{\infty\}$, and
in fact on $\sP^1$ by changing a projective coordinate on $\bP^1$.
The restriction $\phi|U$ is subharmonic 
since $\Delta\phi|U=(\Delta\phi+\mu_f)|U\ge 0$.
By the assumption \eqref{eq:quicknonarch},
the open subset $\{\phi<0\}$ in $\sP^1$ contains $D\setminus\bP^1\neq\emptyset$, and in turn by the maximum principle (cf.\ \cite[Proposition 8.14]{BR10})
applied to the subharmonic and non-positive $\phi|U\le 0$, we have 
\begin{gather*}
 U\subset\{\phi<0\}.
\end{gather*}
Moreover, $\{\phi<0\}\subset \sF(f)$, so that $\phi\equiv 0$ on $\sJ(f)$; for, if 
there is $\cS\in \sJ(f)\cap\{\phi<0\}$, then
for any open neighborhood $D'\Subset\{\phi<0\}$ of $\cS$, 
recalling the uniform convergence \eqref{eq:Green} and 
the equality \eqref{eq:rootsnormalized},
by a version of
Hartogs lemma (cf.\ \cite[Proposition 2.18]{FR09}, \cite[Proposition 8.57]{BR10}) already recalled in the above, we also have
\begin{gather*}
 \limsup_{j\to\infty}\sup_{D'}\frac{\log[f^{n_j},g]_{\can}(\cdot)}{d^{n_j}+\deg g}\le\sup_{D'}\phi<0.
\end{gather*}
Then $D'\subset \sF(f)$ by Lemma \ref{th:Fatou}, which contradicts
$\cS\in \sJ(f)\cap D'$.

Now assume in addition $\infty\in f^{-1}(U)\setminus U(\neq\emptyset)$ (so that 
$U\Subset\sP^1\setminus\{\infty\}$) with no loss of generality,
by changing a projective coordinate 
on $\bP^1$ if necessary. Let us define the function
\begin{gather*}
 \psi:=\begin{cases}
	\phi & \text{on }U\\
	0 & \text{on }\sP^1\setminus U
       \end{cases}:\sP^1\to\bR_{\le 0}\cup\{-\infty\}.
\end{gather*}
The function $\psi+g_F+\log\max\{1,|\cdot|\}$ 
on $\sP^1\setminus\{\infty\}$ is not only 
upper semicontinuous on $\sP^1\setminus\{\infty\}$ (since 
so is $\phi+g_F+\log\max\{1,|\cdot|\}$,
$\psi=\phi=0$ on $\partial U\subset \sJ(f)$, and $\phi\le 0$)
but also subharmonic on 
both $U$ and $\sP^1\setminus(\overline{U}\cup\{\infty\})$, 
and is indeed subharmonic (or equivalently,
{\itshape domination} subharmonic \cite[\S8.2, \S7.3]{BR10})
on $\sP^1\setminus\{\infty\}$ (since so is $\phi+g_F+\log\max\{1,|\cdot|\}$, 
$\psi=\phi=0$ on $\partial U\subset \sJ(f)$, and $\phi\le 0$).
Consequently, also by \eqref{eq:Laplacianaffine},
we have the {\itshape probability} Radon measure
\begin{gather*}
 \Delta\psi+\mu_f=\Delta(\psi+g_F+\log\max\{1,|\cdot|\})+\delta_{\infty}\quad\text{on }\sP^1. 
\end{gather*}

Now suppose to the contrary that $f$ has no potentially good reductions. 
Then we claim that $\Delta\psi=0$ on $\sP^1$;
for, under this assumption, we first have $\mu_f(\partial U)=0$ by 
($f^{-1}(V)\neq V=f^p(V)$ and) Lemma \ref{th:boundarynonarch}. Next,
by the definition of $\psi$, we have $\Delta\psi=0$ on $\sP^1\setminus\overline{U}$,
or equivalently, $\Delta\psi+\mu_f=\mu_f$ on $\sP^1\setminus\overline{U}$. 
Finally, also by 
$U\subset\sP^1\setminus \sJ(f)=\sP^1\setminus\supp\mu_f$ 
and $\mu_f(\partial U)=0$,
we have
\begin{multline*}
 (\Delta\psi+\mu_f)(\overline{U})
=1-(\Delta\psi+\mu_f)(\sP^1\setminus\overline{U})
=1-\mu_f(\sP^1\setminus\overline{U})=\mu_f(\overline{U})=0.
\end{multline*}
Consequently, $\Delta\psi+\mu_f=\mu_f$ on $\sP^1$, 
or equivalently, 
the claim holds.

Once this claim is at our disposal, 
$\psi$ is constant on $\sP^1\setminus\bP^1$, 
which with $\psi:=0$ on $\sP^1\setminus U(\supset f^{-1}(U)\setminus U\neq\emptyset)$
yields $\psi\equiv 0$ on $\sP^1$.
Then we must have $\phi=:\psi\equiv 0$ on $U(\subset\{\phi<0\}$),
which is a contradiction. \qed

\section{Proof of Theorem \ref{th:equidistpositive}}
\label{sec:equidistnonarch}
Let $K$ be an algebraically closed field of any characteristic
that is complete with respect to a non-trivial and non-archimedean
absolute value $|\cdot|$. 
Let $f\in K(z)$ be of degree $d>1$ and $g\in K(z)$ be of degree $>0$,
and fix a lift of $F$.
By \eqref{eq:Green}, \eqref{eq:canonicaldefining}, \eqref{eq:rootsnormalized}, 
and a continuity of the Laplacian $\Delta$,
the equidistribution \eqref{eq:equidistmoving} would follow from
\begin{gather}
 \lim_{n\to\infty}\frac{\log[f^n,g]_{\can}(\cdot)}{d^n+\deg g}=0\quad\text{on }\sP^1\setminus\bP^1.\tag{\ref{eq:equidistmoving}$'$}\label{eq:potentialnonarch}
\end{gather}
Unless \eqref{eq:potentialnonarch} holds,
by an argument similar to that in the previous section
involving a version of Hartogs's lemma (cf.\ \cite[Proposition 2.18]{FR09}, \cite[Proposition 8.57]{BR10}), there exist
a sequence $(n_j)$ in $\bN$ tending to $\infty$ as $j\to\infty$
and a function $\phi:\sP^1\to\bR_{\le 0}\cup\{-\infty\}$ such that
\begin{gather*}
 \phi:=\lim_{j\to\infty}\frac{\log[f^{n_j},g]_{\can}(\cdot)}{d^{n_j}+\deg g}
\quad\text{on }\sP^1\setminus\bP^1,
\end{gather*}
that $\Delta\phi+\mu_f=\Delta(\phi+g_F)+\delta_{\cS_{\can}}$ is 
a {\em probability} Radon measure on $\sP^1$, 
and that $\{\phi<0\}$ is a {\em non-empty and open} subset in $\sP^1$.
For any domain $D'\Subset\{\phi<0\}$, 
recalling the uniform convergence \eqref{eq:Green} and 
the equality \eqref{eq:rootsnormalized},
by a version of Hartogs lemma (cf.\ \cite[Proposition 2.18]{FR09}, \cite[Proposition 8.57]{BR10}), we must have
\begin{gather*}
 \limsup_{j\to\infty}\sup_{\cS\in D'}\frac{\log[f^{n_j},g]_{\can}(\cS)}{d^{n_j}+\deg g}
\le\sup_{D'}\phi<0.
\end{gather*}
This is impossible if $f$ has no potentially good reductions,
by Theorem \ref{th:apriorinonarch}. \qed

\begin{remark}
The difference between the proofs of Theorem \ref{th:equidistpositive}
and Favre--Rivera-Letelier's \cite[Th\'eor\`eme B]{FR09} 
(in the $\operatorname{char}K=0$ case but without 
the no potentially good reductions assumption) is caused by the fact that
when $\operatorname{char}K>0$, no (geometric) structure theorems 
on {\itshape quasiperiodicity domains} 
(containing $g(D)$ in the proof of Theorem \ref{th:apriorinonarch}), 
which are subsets of {\itshape domaines singuliers} (appearing as $V$ in
the proof of Theorem \ref{th:apriorinonarch}), have been known.
\end{remark}

\begin{acknowledgement}
 This research was partially supported by JSPS Grant-in-Aid 
 for Scientific Research (C), 15K04924.
\end{acknowledgement}

\def\cprime{$'$}

\end{document}